\documentclass[11pt]{amsart}

\usepackage{amssymb,amsthm,amsmath,epsfig,latexsym}
\usepackage{calc}
\usepackage{pgf,tikz}
\usetikzlibrary{arrows}
\usepackage{mathrsfs}
\usepackage{geometry} 
\usepackage{multicol}
\usepackage{array}

\newcommand{\C}{{\mathcal{C}}}

\newcommand{\hyp}{\operatorname{hyp}}
\newcommand{\ev}{\operatorname{ev}}
\newcommand{\rk}{\operatorname{rk}}
\newcommand{\HF}{\operatorname{HF}} 
\newcommand{\reg}{\operatorname{reg}}

\theoremstyle{plain}

\newtheorem*{thm*}{Theorem}
\newtheorem{defn0}{Definition}[section]
\newtheorem{prop0}[defn0]{Proposition}
\newtheorem{quest0}[defn0]{Question}
\newtheorem{thm0}[defn0]{Theorem}
\newtheorem{lem0}[defn0]{Lemma}
\newtheorem{corollary0}[defn0]{Corollary}
\newtheorem{example0}[defn0]{Example}
\newtheorem{remark0}[defn0]{Remark}
\newtheorem{conj0}[defn0]{Conjecture}

\newenvironment{defn}{\begin{defn0}\rm}{\end{defn0}}
\newenvironment{prop}{\begin{prop0}}{\end{prop0}}

\newenvironment{thm}{\begin{thm0}}{\end{thm0}}
\newenvironment{lem}{\begin{lem0}}{\end{lem0}}
\newenvironment{cor}{\begin{corollary0}}{\end{corollary0}}
\newenvironment{ex}{\begin{example0}\rm}{\end{example0}}
\newenvironment{rem}{\begin{remark0}\rm}{\end{remark0}}

\begin{document}


\title{Geometry of the Minimum Distance}
\author{John Pawlina and \c{S}tefan O. Toh\v{a}neanu}

 \keywords{evaluation codes, generalized projective Reed-Muller codes, minimum distance, socle degrees, initial degree, free resolution. \\ \indent Department of Mathematics and Statistical Science, University of Idaho, Moscow, Idaho 83844-1103, USA, Email: jpawlina@uidaho.edu, tohaneanu@uidaho.edu, Phone: 208-885-6234, Fax: 208-885-5843.}

\begin{abstract}
	Let \({\mathbb K}\) be any field, let \(X\subset {\mathbb P}^{k-1}\) be a set of \(n\) distinct \({\mathbb K}\)-rational points, and let \(a\geq 1\) be an integer. In this paper we find lower bounds for the minimum distance \(d(X)_a\) of the evaluation code of order \(a\) associated to \(X\). The first results use \(\alpha(X)\), the initial degree of the defining ideal of \(X\), and the bounds are true for any set \(X\). In another result we use \(s(X)\), the minimum socle degree, to find a lower bound for the case when \(X\) is in general linear position. In both situations we improve and generalize known results.
\end{abstract}

\maketitle



\section{Introduction}\label{sec_intro}

	A loose definition of evaluation codes is that they are linear codes obtained by evaluating the elements of a finite-dimensional \({\mathbb K}\)-vector space, \(V\), of rational functions/ polynomials on a finite set of affine or projective points, \(X\). For example, for algebraic-geometric codes (AG-codes), \(V\) is the Riemann-Roch space and \(X\) is a divisor on an irreducible curve; toric codes have a similar definition. Let \(R = {\mathbb K}[x_1,\dots,x_k]\) be the ring of polynomials over \({\mathbb K}\). For the classical Reed-Solomon and (generalized) Reed-Muller codes, \({\mathbb K} = {\mathbb F}_q\), while \(V\) is the vector space of all polynomials either of degree at most \(a\), i.e. \(V = R_{\leq a}\), or of degree exactly \(a\), i.e. \(V = R_a\), and the finite sets are the sets of all \({\mathbb F}_q\)-rational points of the ambient affine or projective space. From this last description, evaluation codes are now known to be the codes where \(V\) is chosen to be a subspace of \(R_{\leq a}\) or \(R_a\) and \(X\) is a finite set of points in affine or projective space. Also, \({\mathbb K}\) may be taken to be any field. We found \cite{JaVaVi} to be a good reference.

	Our introduction to evaluation codes is the description that is found in \cite{Ha1}: \(V = R_a\) and \(X\) is a finite set of distinct \({\mathbb K}\)-rational points in \({\mathbb P}^{k-1}\). Some authors call this type of code a generalized projective Reed-Muller code, while others may refer to an affine version of this kind of code as a Reed-Muller-type code.

	Numerous researchers have been analyzing the parameters of evaluation codes, especially because of their strong connection with commutative and homological algebra: if \(\C(X)_a\) is the evaluation code of order \(a\) associated to a finite set of points \(X\), then the length of \(\C(X)_a\) is the degree of the defining ideal of \(X\), the dimension equals the Hilbert function of the ring of regular functions on \(X\) (or coordinate ring of \(X\)) evaluated at \(a\), and the minimum (Hamming) distance \(d(X)_a\) of \(C(X)_a\) is strongly connected with the maximum cardinality of a proper subset of \(X\) contained in a hypersurface of degree \(a\). This geometric interpretation of the minimum distance led to the title of this article, but most importantly it leads immediately to the famous Cayley-Bacharach Theorem. In Section \ref{sec_evalcodes} we describe in more detail what we mentioned in this paragraph.

	There are two general approaches to studying evaluation codes: The first is more computational, where the sets \(X\) are somewhat specific and we can work in a hands-on way with them. See \cite{JaVaVi} for a good example of this type of research. The second approach uses homological invariants to understand the minimum distance. Of course, both approaches are useful for discovering and proving theorems and there are overlaps between them. In this paper we will pursue the second approach, and for details we have Section \ref{sec_commalg} for some brief preliminaries.
	
	Let \(s(X)\) be the minimum socle degree. This is a number that can be obtained from the minimum shift in the ``tail'' of a graded minimal free resolution of the coordinate ring of \(X\), i.e., \({\bf F}_{k-1}\) in the free resolution (\ref{eqn_free_res}) in \S \ref{sec_commalg} below. For the case when the coordinate ring is Gorenstein (such as for any complete intersection), \(s(X)\) equals the Castelnuovo-Mumford regularity. Under certain conditions, it has been discovered that \(d(X)_a\geq s(X)-a+1\). This result was obtained in \cite{Ha2}, in \cite{GoLiSc}, and in \cite{To4} when \(X\) is a finite set of points in \({\mathbb P}^2\) whose defining ideal is a complete intersection, then more generally under the same conditions in \({\mathbb P}^{k-1}\) for \(k\geq 3\), and finally in \({\mathbb P}^{k-1}\) for \(k\geq 3\) when the coordinate ring is Gorenstein, respectively. In \cite{To4}, the same bound is obtained for {\em any} set of points in \({\mathbb P}^2\) or \({\mathbb P}^3\). The same bound is obtained in \cite{To5} for any set of points in \({\mathbb P}^{k-1}\) for \(k\geq 3\), but under the condition that the base field is of zero characteristic. This condition is imposed by the theoretical results used in the proof, and it is not known if there are different proofs which will allow dropping this condition, which is very restrictive from a coding theory or computational point of view.

	For any finite set of points \(X\subset{\mathbb P}^{k-1}, k\geq 3\), and any base field, \cite{ToVa} obtained the bound \(d(X)_1\geq \alpha(X)-1\), where \(\alpha(X)\) is the initial degree of \(X\), which is the minimum degree of a generator of the defining ideal of \(X\). Compared to the bound \(d(X)_a\geq s(X)-a+1\) (true for characteristic zero, and conjectured in general), when \(a=1\), the former bound is not as good since \(s(X)\geq \alpha(X)-1\), with equality in very special cases.

	The article \cite{BaFo} presented a bound which has been the inspiration for the work in our article. The authors demonstrated that if \(X\subset{\mathbb P}^{k-1}, k\geq 3,\) is a finite set of points in general linear position whose defining ideal is a complete intersection, then \(d(X)_a\geq (k-1)(s(X)-a-1)+2\). What drew our attention was the coefficient \(k-1\) in front of \(s(X)\), which provides a much better bound than those we mentioned above.

	In our article we obtain the following results. In Corollary \ref{cor_alphabound}, Theorem \ref{thm_better}, and Proposition \ref{prop_bound_comparison} we demonstrate for any finite set of points \(X\subset {\mathbb P}^{k-1}, k\geq 3\), and any integer \(1\leq a\leq \alpha(X)-1\) we have
		\[
			d(X)_a\geq {{\alpha(X)-1-a+k-1}\choose{k-1}}\geq (k-1)(\alpha(X)-1-a)+1.
		\]
In Theorem \ref{thm_socle}, for any finite set of points \(X\subset{\mathbb P}^{k-1}, k\geq 3\), in general linear position (so not necessarily also a complete intersection), we have either
		\[
			d(X)_a\leq k-1 \ \ \ \mbox{or} \ \ \ d(X)_a\geq (k-1)(s(X)-1-a)+2.
		\]
The geometry of the points \(X\) becomes transparent in the proofs of Corollary \ref{cor_alphabound} and Theorem \ref{thm_socle}: the key fact is that \(n\) points in \({\mathbb P}^{k-1}\) can be all placed on \(\lceil n/(k-1)\rceil\) hyperplanes.

	One may question the legitimacy of using classical algebraic geometry techniques when working over a finite base field \({\mathbb K}\) which is, of course, not algebraically closed. For the theoretical results presented in this paper, we do not require this condition: The ideals of points \(I(X)\), above, are intersections of linear prime ideals in \(R:={\mathbb K}[x_1,\ldots,x_k]\). Via the inclusion \({\mathbb K}\subseteq \overline{\mathbb K}\), the algebraic closure, we can think of the points of \(X\) as having coordinates in \(\overline{\mathbb K}\), and \(I(X)\subset\overline{\mathbb K}[x_1,\ldots,x_k]\), but, because \(X\) is already a \({\mathbb K}\)-rational variety, the geometric features of \(X\) coming from the coding theory remain the same under this embedding. We can say also that the algebraic geometric properties of \(I(X)\) are not affected by this embedding, since the primary decomposition of \(I(X)\) has all components linear prime ideals defining the \({\mathbb K}\)-rational points of \(X\). Furthermore, if we look at the commutative algebraic / homological features of \(I(X)\), in any instance one always extends the coefficients to be regarded in \(\overline{\mathbb K}\). For more details, see the proof of Theorem 4.10 in \cite{CSTVV}.

	The issues with the base field not being algebraically closed usually occur if we want to look at the problem from the reverse point of view, or if we want to give some intuitive examples. For example, if we describe the ideal \(I(X)\) in terms of its generators, we may have issues if \({\mathbb K}\) is not algebraically closed. More concretely, if we want \(X\subset {\mathbb P}_{\mathbb K}^2\) to be the transverse intersection of a conic curve and a cubic curve (so \(I(X)\) is a complete intersection), then obviously here we apply B\'{e}zout's Theorem to obtain that \(X\) consists of 6 points, but the coordinates of these points are in \(\overline{\mathbb K}\). Or, as we will see in some examples, we will be choosing many points in \({\mathbb P}_{\mathbb K}^{k-1}\), and if the field \({\mathbb K}\) does not have enough elements, this selection is indeed impossible.


\section{Preliminaries}\label{sec_prelim}

	Let \(\C\) be a linear code with minimum distance \(d\) and with generator matrix \(G\) of size \(k\times n\), with \(n\geq k\geq 1\), which has full rank \(k\) and no zero columns (i.e., \(\C\) is {\em nondegenerate}). We call \(\C\) an \([n,k,d]\)-linear code. In this paper we will suppose that no two columns of \(G\) are proportional. That is, we consider the set of \(n\) points \(X_{\C}:=\{P_1,\ldots,P_n\}\subset{\mathbb P}^{k-1}\), where the homogeneous coordinates of the point \(P_i\) are the entries of the \(i\)-th column of \(G\). When there is no risk of confusion, the subscript \(\C\) will be omitted from \(X_\C\).
	
	We will say that {\em \(\C\) is defined by the set \(X_{\C}\)}; of course, this ``definition'' is modulo a permutation of the order in which the points are listed. Nonetheless, different orders and different representatives of the coordinates of the points lead to ``monomially equivalent codes''.  Two \([n,k]\)-linear codes with \(k\times n\) generator matrices \(G_1\) and \(G_2\) are {\em monomially equivalent} if and only if there exists an \(n\times n\) invertible matrix \(M\), which is the product of a diagonal matrix and a permutation matrix, such that \(G_1=G_2M\). Two monomially equivalent codes have the same parameters \(n, k\), and \(d\).
	
	For any finite set of points \(Y\subset{\mathbb P}^{k-1}\), define \(\hyp(Y)\) to be the maximum number of the points of \(Y\) that are contained in a hyperplane of \({\mathbb P}^{k-1}\). Note that \(X_C\) is not all contained in a hyperplane because \(\operatorname{rk}(G) = k\). That is, \(\hyp(X_C)<|X|\). The following is a well-known result; see, for example, Part 1.1.2 of \cite{Tsfasman} for details.

\begin{prop}\label{prop_mindistpoints}
	Let \(\C\) be an \([n,k,d]\)-linear code with defining set \(X_{\C}\subset{\mathbb P}^{k-1}\). Then the minimum distance \(d\) of \(\C\) satisfies \(d=n-\hyp(X_{\C})\).
\end{prop}


\subsection{Evaluation Codes.} \label{sec_evalcodes}

	Recall if \(P=[a_1,\ldots,a_k]\in{\mathbb P}^{k-1}\) is a point, then the ideal of \(P\) in \(R:={\mathbb K}[x_1,\ldots,x_k]\) is \(I(P)=\langle\{a_ix_j-a_jx_i | i,j\in\{1,\ldots,k\}, i\neq j\}\rangle\). Let \(X=\{P_1,\ldots,P_n\}\subset{\mathbb P}^{k-1}\) be a finite set of points. The defining ideal of \(X\) is \(I(X)=I(P_1)\cap\cdots\cap I(P_n)\subset R\). For an integer \(a\geq 1\), we let \(R_a\) be the \({\mathbb K}\)-vector space of homogeneous polynomials in \(R\) of degree \(a\).
	
	Suppose we choose the standard representatives for the homogeneous coordinates of the points \(P_i\) (i.e., the leftmost nonzero coordinate of each point is 1). This allows us to make a well-defined \({\mathbb K}\)-linear map:
	\[
		\ev_a: R_a\longrightarrow {\mathbb K}^n \ \text{by} \ \ev_a(f)=(f(P_1),\ldots, f(P_n)).
	\]
	We can also define \(\ev_a\) without relying on a standard representation of the points, by first choosing \(f_0\in R_a\) which does not vanish at any of the points \(P_i\), then defining \(\displaystyle \ev_a(f)=\left(\frac{f(P_1)}{f_0(P_1)},\ldots,\frac{f(P_n)}{f_0(P_n)}\right)\) for \(f\in R_a\). What follows is independent of the choice of \(f_0\).

	The image of \(\ev_a\), denoted \(\mathcal C(X)_a\), is {\em the evaluation code of degree/order \(a\) associated to \(X\)}. The parameters of \(\mathcal C(X)_a\) are captured in the following proposition:

\begin{prop} \label{prop_paramEvaluation}
	Let \(X\subset{\mathbb P}^{k-1}\) be a finite set of \(n\) points, and let \(a\geq 1\) be an integer. Let \(I(X)\subset R:={\mathbb K}[x_1,\ldots,x_k]\) be the defining ideal of \(X\). Then, the basic parameters of \(C(X)_a\), the evaluation code of degree \(a\) associated to \(X\), are the following:
	\begin{itemize}
		\item Length: \(n=|X|=\deg(I(X))\).
		\item Dimension: \(k(X)_a:=\HF(R/I(X),a)\), the Hilbert function of \(R/I(X)\) evaluated at \(a\).
 		\item Minimum distance: \(\displaystyle d(X)_a:=n-\max_{X'\subset X}\{|X'|\,|\, \dim_{\mathbb K}(I(X')_a)\gneq\dim_{\mathbb K}(I(X)_a)\}\).
	\end{itemize}
\end{prop}

\begin{proof}
	The first item is immediate: the cardinality of a finite set of points is the degree of the defining ideal of the set. (See Lemma \ref{lem_reminders}(i).)

	For the second statement, \(R_a/\ker(\ev_a)\simeq C(X)_a\). At the same time, \(f\in\ker(\ev_a)\) if and only if \(f\in I(X)_a\). The claim follows from the definition of the Hilbert function.

	The third statement was observed first in \cite[Proposition 6]{Ha1}. Let \({\bf w}\in C(X)_a\) be a codeword of weight \(s\geq 1\). Then \({\bf w}=(f(P_1),\ldots,f(P_n))\), where \(f\in R_a\) with \(f(P_{i_1}),\ldots,f(P_{i_s})\neq 0\) and \(f(P_j)=0\) for all \(j\in\{1,\ldots,n\}\setminus\{i_1,\ldots,i_s\}\). Let \(X':=\{P_j| j\in\{1,\ldots,n\}\setminus\{i_1,\ldots,i_s\}\}\). We have \(X'\subsetneq X\), \(f\in I(X')_a\setminus I(X)_a\), and \(s=n-|X'|\). Since \(X'\subset X\), then \(I(X)_a\subset I(X')_a\), with strict containment because of \(f\). Therefore \(\dim(I(X')_a)\gneq\dim(I(X)_a)\). If we want to minimize \(s\geq 1\), we must maximize \(|X'|\), where \(X'\) has the properties listed previously.
\end{proof}

\begin{rem} \label{rem_evalcodes}
	The following facts are useful to better understand evaluation codes and their minimum distances.
	\begin{itemize}
		\item[(1)] If we define \(\hyp(X)_a:= \max_{X'\subset X}\{|X'|\,|\, \dim_{\mathbb K}(I(X')_a)\gneq\dim_{\mathbb K}(I(X)_a)\}\), then \(\hyp(X)_a\) is the maximum size of a proper subset of \(X\) which is contained in a hypersurface of degree \(a\). Notice that \(\hyp(X)_1 = \hyp(X)\) when \(X\) is not contained in a hyperplane.
	
		\item[(2)] In the spirit of the Cayley-Bacharach theorems (see \cite{EiGrHa}), the geometric interpretation of \(d(X)_a\) and \(\hyp(X)_a\) is the following: any hypersurface of degree \(a\) which passes through at least \(\hyp(X)_a+1 = |X|-d(X)_a+1\) points of \(X\) should pass through \textit{all} points of \(X\).

		\item[(3)] For \(a=1\), if we use the (standard) \({\mathbb K}\)-basis \(\{x_1,\ldots,x_k\}\) of \(R_1\), the matrix representation of \(\ev_1\) with respect to this basis is exactly the matrix having as columns the coordinates of the points of \(X\). Hence \(d(X)_1\) is the minimum distance of the linear code defined by \(X\).

		\item[(4)] Generalizing item (3) above, for \(a\geq 1\), the matrix of \(\ev_a\) in standard bases of \(R_a\) and \({\mathbb K}^n\) has columns the homogeneous coordinates of the \(n\) points of \(v_a(X)\subset {\mathbb P}^{N_a-1}\), where \(\displaystyle N_a:=\dim_{\mathbb K}R_a={{k-1+a}\choose{a}}\) and \(v_a:{\mathbb P}^{k-1}\longrightarrow {\mathbb P}^{N_a-1}\) is the Veronese embedding. In order to apply Proposition \ref{prop_mindistpoints} to compute \(d(X)_a\), first we must embed \(v_a(X)\) into the smaller dimensional projective space \({\mathbb P}^{k(X)_a-1}\), where \(v_a(X)\) is nondegenerate (i.e., is not all contained in a hyperplane of this smaller projective space).
	\end{itemize}
\end{rem}


\subsection{Some Homological/Commutative Algebra.} \label{sec_commalg}

	We now present some homological/commutative algebraic concepts that we will use throughout the remainder of this paper.

	Let \(R = \bigoplus_{i\geq 0} R_i\) be a graded ring. Suppose \(M=\bigoplus_{i\geq 0}M_i\) is a finitely generated, (positively) graded \(R\)-module. The {\em initial degree} of \(M\) is \(\displaystyle\alpha(M):=\min\{i|M_i\neq 0\}\). If \(X\subset{\mathbb P}^{k-1}\) is a finite set of points, then we denote
	\[
		\alpha(X):=\alpha(I(X)),
	\]
where \(I(X)\) is the defining (homogeneous) ideal of \(X\) in the graded ring \(R := {\mathbb K}[x_1,\ldots,x_k]\), with the standard grading (by the degree). In this context, \(\alpha(X)\) is the smallest degree among nonzero elements of \(I(X)\).

	Let \(A\) be a Noetherian graded ring and suppose \(J\subset A\) is a homogeneous ideal of maximal height, i.e., \(A/J\) is an Artinian ring. If \({\mathfrak m}\) is the irrelevant ideal of \(A\), define {\em the socle of \(A/J\)} to be the graded \(A/J\)-module \(\displaystyle {\rm Soc}(A/J):=\frac{J:{\mathfrak m}}{J}\), and {\em the minimum socle degree} to be \(s(A/J):=\alpha({\rm Soc}(A/J))\).

	Let \(X=\{P_1,\ldots,P_n\}\subset{\mathbb P}^{k-1}\) be a finite set of points. Let \(I:=I(X)\); then \({\rm ht}(I)=k-1\). Let \(L\in R_1\) be a linear form such that \(L(P_i)\neq 0\) for all \(i=1,\ldots,n\), meaning \(L\) is a non-zerodivisor on \(R/I\).\footnote{By \cite{Schenck}, \(f\in R\) is  a \textit{non-zerodivisor} on the \(R\)-module \(M\) if \(f\cdot m\neq 0\) for all nonzero \(m\in M\). Other sources call such an \(f\) an \textit{\(M\)-regular element} of \(R\).} If needed (i.e., if \({\mathbb K}\) is too small), we can safely choose \(L\) to have coefficients in \(\overline{\mathbb K}\); see the comments we made at the end of the Introduction.

Let \(A = R/\langle L\rangle\) and let \(J = \langle I, L\rangle/ \langle L\rangle\). Since \(A/J\simeq R/\langle I,L\rangle\) as rings, and \({\rm ht}(\langle I,L\rangle) = k = \operatorname{dim}(R)\), then \(A/J\) is an Artinian ring called {\em an Artinian reduction of \(I\)}. Then we define {\em the minimum socle degree of \(X\)} to be
	\[
		s(X):=s(A/J).
	\]
See \cite{MigPat} for more details on Artinian reductions. The takeaway is that \(R/I\) as an \(R\)-module and \(A/J\) as an \(A\)-module have the same graded Betti numbers.

	The minimum socle degree \(s(X)\) can be also obtained in the following way: Because \(R/I\) is arithmetically Cohen-Macaulay, \(R/I\) has a minimal graded free resolution of length \(k-1\),
	\begin{align}
		0&\rightarrow {\bf F}_{k-1}\rightarrow {\bf F}_{k-2}\rightarrow\cdots\rightarrow {\bf F}_1\rightarrow R\rightarrow R/I\rightarrow 0, \label{eqn_free_res}
	\end{align}
where each \({\bf F}_i\) is a graded free \(R\)-module. If the leftmost nonzero module is
	\[
		{\bf F}_{k-1}=\bigoplus_{j=1}^tR(-b_j),
	\]
with \(k-1\leq b_1\leq\cdots\leq b_t\), then
	\[
		s(X)=b_1-k+1 = b_1 - (k-1).
	\]
Details of the proof of this equality can be found, for example, in \cite[Lemma 2.61]{ToB}.

	The number \(\reg(X):=b_t-k+1\) is the {\em (Castelnuovo-Mumford) regularity of \(X\)}, and one has that \(\HF(R/I,\reg(X))=\deg(I)=|X|=n\) and \(\HF(R/I,\reg(X)-1)<\deg(I)\). See \cite[Theorem 4.2 Part (3)]{EiGeo} for details.

	We have the inequalities
	\[
		\reg(X)\geq s(X)\geq \alpha(X)-1.
	\]
	The first inequality is immediate, and the second inequality is true because we start with a \textit{minimal} free resolution, and therefore the minimum shift at \({\bf F}_i\) is strictly bigger than the minimum shift at \({\bf F}_{i-1}\).

\medskip

\begin{defn}\label{defn_generic}
	Let \(X=\{P_1,\ldots, P_n\}\) be a set of \(n\) points in \({\mathbb P}^{k-1}\). Let \(R:={\mathbb K}[x_1,\ldots,x_k]\) and let \(I = I(X)\) be the defining ideal of \(X\).
	\begin{itemize}
		\item[(1)] \(X\) is said to be {\em in generic \(n\)-position}, or simply {\em in generic position}, if
						\[
							\HF(R/I,i)=\min\left\{n,{{i+k-1}\choose{k-1}}\right\}
						\]
				for all \(i\geq 0\).

		\item[(2)] \(X\) is said to be {\em in general linear position} if any \(u\leq k\) points of \(X\) will span a \({\mathbb P}^{u-1}\); equivalently, any representatives of the homogeneous coordinates of any \(u\leq k\) of the points of \(X\) are linearly independent (affine) vectors in \({\mathbb K}^k\).

	\end{itemize}
\end{defn}

	Sets of points \(X\) in generic position or general linear position have been studied before. However, despite the apparent ``simple'' geometric/algebraic features, some information about these kinds of sets still elude interested researchers. The best example is what is know as ``The Minimal Resolution Conjecture'' (\cite{Lo2}). Also, Problems (A) and (B) in \cite{GeMa} are the best reflection of the entire discussion around these issues. In the remark below we will present some of the information that is known about such \(X\).

\begin{rem}\label{rem_generic}
	Let \(X=\{P_1,\ldots, P_n\}\) be a set of \(n\) points in \({\mathbb P}^{k-1}\), not all contained in a hyperplane.
	\begin{itemize}
		\item[(i)] \(X\) is in general linear position if and only if the linear code with generator matrix having as columns any representatives of the coordinates of the points of \(X\) is a Maximum Distance Separable (MDS) code (i.e., an \([n,k,d] = [n,k,n-k+1]\)-linear code).

		\item[(ii)] As \cite[Remarks 1.7 (3)]{GeMa} mentions, there exist sets of points \(X\) that are in general linear position but not in generic position, and also sets of points in generic position that are not in general linear position. An example of the latter is a set \(X\) of 5 points in \({\mathbb P}^2\), where 3 of the points are on a line, and the other 2 are off that line. Because of this extra collinearity, \(X\) is not in general linear position, but from the Hilbert function values
				\[
					H(R/I(X),i) =		\begin{cases}
										1 = {{0+2}\choose{2}} &\mbox{for}\ \ i = 0,\\
										3 = {{1+2}\choose{2}} &\mbox{for}\ \ i = 1,\\
										5 = |X| &\mbox{for}\ \ i \geq 2
									\end{cases}
				\]
				we have that \(X\) is in generic position.

		\item[(iii)] Let \(\C\) be an \([n,k,d]\)-linear code with defining matrix \(G\) and corresponding set of points \(X = X_\C\subset {\mathbb P}^{k-1}\) as described at the start of Section \ref{sec_prelim}. For \(X\) to be in general linear position, Definition \ref{defn_generic} (2) says at least one maximal minor of \(G\) is nonzero. A maximal minor of \(G\) being nonzero means that the point representing \(G\) in \({\mathbb P}^{kn-1}\) belongs to the complement of the variety defined by the vanishing of that maximal minor, i.e. an open subset. When working over the algebraic closure, hence an infinite field, the complement of a variety is dense in the Zariski topology.
		
		\item[(iv)] Suppose \(X\) is a set of \(\displaystyle {{m-1+k-1}\choose{k-1}}\) points in generic position. Then, by \cite[Proposition 1.4]{GeMa}, \(\alpha(X)=m\) and the minimum number of generators of \(I = I(X)\) is \(\displaystyle \mu(I)={{m-1+k-1}\choose{k-2}}\). By \cite[Exercise 7.1.9]{Schenck} or \cite[Theorem 2.2]{Lo1}, \(\reg(X)=m-1\). With this, from the inequality \(\reg(X)\geq \alpha(X)-1\) we have seen above, we must have \(\reg(X)=\alpha(X)-1=m-1\), hence \(I\), and therefore \(R/I\), has linear graded free resolution.
		
		\item[(v)] A finite set of points \(X\subset {\mathbb P}^{k-1}\) is in \textit{uniform position} if \(X\) and each subset of \(X\) is in generic position. By \cite[Corollary 14]{Ha1}, for such a set \(X\), \(C(X)_a\) is an MDS code, meaning
			\[
				d(X)_a =	\left|X\right| - \HF(R/I(X),a)+1	=	\begin{cases}
														\left|X\right| - {a+k-1\choose k-1}+1, &\mbox{if } a < \reg(X),\\
														1, &\mbox{if } a\geq \reg(X).
													\end{cases}
			\]
	\end{itemize}
\end{rem}

\medskip

	Another very important homological invariant associated to a finite set of points is {\em the \({\bf v}\)-number} (this can be defined in general for any homogeneous ideal; see \cite[Section 4.1]{CSTVV}): if \(X=\{P_1,\ldots,P_n\}\subset {\mathbb P}^{k-1}\) is a finite set of points with defining ideal \(I:=I(X)\subset R:={\mathbb K}[x_1,\ldots,x_k]\), then
	\[
		{\bf v}(X):=\min\left\{\left.\alpha\left(\dfrac{I:I(P_i)}{I}\right)\right|i=1,\ldots,n\right\},
	\]
where \(I(P_i)\) is the ideal of the point \(P_i\in X\).\footnote{This definition is in fact \cite[Proposition 4.2]{CSTVV}.}

	For each \(i=1,\ldots,n\), the number \(\alpha((I:I(P_i))/I)\) is called {\em the minimum degree of a separator of the point \(P_i\)} (see \cite{GeKrRo}). If the characteristic of \({\mathbb K}\) is zero, by \cite[Theorem 3.3]{ToVa}, which is an immediate consequence of \cite[Theorem 5.4]{GuMaVa}, we have \({\bf v}(X)\geq s(X)\).\footnote{The characteristic zero restriction was necessary in \cite[Theorem 5.4]{GuMaVa} to show that the minimum degrees of separators of the points \(P_i\) can be recovered from the shifts in the tail \({\bf F}_{k-1}\).} However, by \cite[Theorem 4.10]{CSTVV}, one can drop the characteristic zero restriction:
	\[
		\reg(X)\geq {\bf v}(X)\geq s(X).
	\]

\begin{rem}\label{rem_vnumber}
	Let \(X\subset{\mathbb P}^{k-1}\), with \(k\geq 3\), be a set of \(n\) points. Then:
	\begin{itemize}
		\item[(i)] Let \(a\geq 1\) be an integer. By Proposition \ref{prop_paramEvaluation}, from the Singleton bound we have \(1\leq d(X)_a\leq |X|-\HF(R/I(X),a)+1\). So, if \(a\geq \reg(X)\), we obtain that \(d(X)_a=1\).
  		\item[(ii)] Even more generally, by \cite[Proposition 4.6]{CSTVV}, \({\bf v}(X)\) is the smallest integer \(a\) such that \(d(X)_{a}=1\). So, by \cite[Proposition 2.1]{To4}, one has
  			\[
  				d(X)_1>d(X)_2> \cdots >d(X)_{{\bf v}(X)-1}>d(X)_{{\bf v}(X)}=\cdots=d(X)_{\reg(X)}=\cdots \,=1.
			\]
			Though we will not make use of \({\bf v}(X)\) in this paper, this last remark should be convincing enough that this homological invariant is very important in the theory of evaluation codes.
	\end{itemize}
\end{rem}


\section{The minimum distance and the initial degree of points} \label{sec_mindistalpha}

	To begin, we provide a lemma that summarizes known results. This lemma will be useful later in the section.

\begin{lem}\label{lem_reminders}
	Let \(Z\subseteq {\mathbb P}^{k-1}\) be a finite set of points and let \(R = {\mathbb K}[x_1,\dots,x_k]\) be the ring of polynomials in \(x_1,\dots,x_k\) over the field \({\mathbb K}\). Then the Hilbert function and initial degree have the following properties.
	\begin{itemize}
		\item[(i)] For \(i >> 0\), the Hilbert function satisfies \(\HF(R/I(Z), i) = |Z|\).
		\item[(ii)] The Hilbert function \(\HF(R/I(Z), i)\) is nondecreasing in \(i\).
		\item[(iii)] If \(Z'\subsetneq Z\), then \(\alpha(Z')\leq \alpha(Z)\leq \alpha(Z')+\alpha(Z\setminus Z')\).
	\end{itemize}
\end{lem}

\begin{proof}
	(i) This is a classic result; for example, this is Exercise 2.3.5 in \cite{Schenck}. Proof when \(|Z| = 1\) is nearly immediate. The proof is then completed by induction on \(|Z|\). Key to the proof is the observation that if \(Z = \{P_1,\dots,P_m\}\), so that \(I(Z) = \bigcap_{j = 1}^m I(P_j)\), and if \(J = \bigcap_{j=1}^{m-1} I(P_j)\), then the sequence
			\[
				0 \to R/I(Z) = R/(J\cap I(P_m)) \to (R/J)\oplus (R/I(P_m)) \to R/(J+I(P_m)) \to 0
			\]
			is exact.

	(ii) This claim is true for all Cohen-Macaulay varieties \(Z\). The specific case in the statement of the lemma is seen in \S 3.3 of \cite{Schenck}. Key to this proof is the fact that for a non-zerodivisor \(f\) on \(R/I(Z)\), we have \(I(Z): f = I(Z)\).
	
	(iii) If \(Z'\subsetneq Z\) then \(I(Z)\subseteq I(Z')\), thus \(\alpha(Z')\leq \alpha(Z)\). For the second part on the inequality, choose \(f\in I(Z')\) with minimal degree \(\alpha(Z')\) and choose \(g\in I(Z\setminus Z')\) with minimal degree \(\alpha(Z\setminus Z')\). Then \(fg\in I(Z)\), so \(\alpha(Z)\leq \deg(fg) = \alpha(Z')+\alpha(Z\setminus Z')\).
\end{proof}

\begin{prop}\label{prop_alpha}
	Let \(X\subset{\mathbb P}^{k-1}, k\geq 3\), be a set of \(n\) points, not all contained in a hyperplane (so \(\alpha(X)\geq 2\)). Let \(1\leq a\leq \alpha(X)-1\). Then
	\[
		d(X)_a+u\geq(k-1)(\alpha(X)-a)
	\]
for some \(u\in\{0,\ldots,k-2\}\).
\end{prop}

\begin{proof}
	Suppose \(X=\{P_1,\ldots,P_n\}\). Let \(m:=d(X)_a\geq 1\). By Proposition \ref{prop_paramEvaluation}, let \(X'\subsetneq X\) be such that \(|X'|=n-m\), let \(Y:=X\setminus X'=\{Q_1,\ldots,Q_m\}\), and let \(f\in I(X')\) with \(\deg(f)=a\) and \(f(Q_j)\neq 0\) for all \(j=1,\ldots,m\) (from the maximality of \(|X'|\)).

	If \(m\leq k-1\), then \(Y\) will be contained in a hyperplane \(V(L)\), and so \(f\cdot L\in I(X)\), giving that \(\alpha(X)\leq a+1\), i.e. \(\alpha(X)-1\leq a\). Thus \(a = \alpha(X)-1\) and the claim is satisfied by using \(u = k-2\).

	If \(m\geq k\), set \(\displaystyle \delta := \lceil m/(k-1)\rceil\). Note \(\delta\cdot(k-1) = m+u\) for some \(u\in \{0,\ldots,k-2\}\). Any \(k-1\) points of \(Y\) will belong to a hyperplane, so there will be a union of \(\delta\) hyperplanes \(V(L_1),\ldots,V(L_{\delta})\subset{\mathbb P}^{k-1}\) which contains the points of \(Y\). This tells us that the product \(f\cdot L_1\cdots L_{\delta}\) belongs to \(I(X)\), and so \(\delta+a\geq \alpha(X)\). Hence, in this case, \(m+u\geq (k-1)(\alpha(X)-a)\) for some \(u\in\{0,\ldots,k-2\}\).
\end{proof}

	Since \(u\leq k-2\), and \(d(X)_a\geq 1\), we have the following corollary. One should note that if \(a\geq \alpha(X)\), then the bound in this corollary becomes \(d(X)_a\geq 0\), which is trivially satisfied. Therefore, in the statement below we decided not to bound the range of the values for \(a\) from above.

\begin{cor}\label{cor_alphabound}
	Let \(X\subset{\mathbb P}^{k-1}, k\geq 3\), be a set of \(n\) points, not all contained in a hyperplane. Let \(1\leq a \) be an integer. Then,
	\[
		d(X)_a\geq (k-1)(\alpha(X)-1-a)+1.
	\]
\end{cor}

\begin{ex}\label{ex_boundcomp1}
	Consider the set of \(10\) points \(X= \{P_1,\dots,P_{10}\}\subset {\mathbb P}^2\) whose coordinates make up the columns of the matrix \(G\) below.

	\begin{center}
		\begin{tabular}{@{}m{\dimexpr.5\textwidth-.5\columnsep}m{\dimexpr.5\textwidth-.5\columnsep}@{}}
			\( \displaystyle
				G =	\begin{pmatrix}
						0 & 1 & 2 & 3 & 0 & 1 & 2 & 3 & 1 & 2\\
						0 & 0 & 0 & 0 & 3 & 3 & 3 & 2 & 1 & 1\\
						1 & 1 & 1 & 1 & 1 & 1 & 1 &1 & 1 & 1
					\end{pmatrix}
			\)
			&
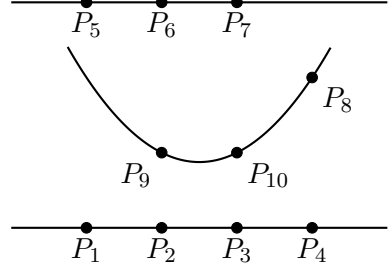
			\begin{tikzpicture}
				\draw[smooth,thick,black,domain=-0.25:3.25,samples=400]
					plot (\x,{.5*(\x*\x-3*\x+4)});

				\draw[thick] (-1,0) -- (4,0);
				\draw[thick] (-1,3) -- (4,3);

				\filldraw[black] (0,0) circle (2pt) node[below]{\(P_1\)};
				\filldraw[black] (1,0) circle (2pt) node[below]{\(P_2\)};
				\filldraw[black] (2,0) circle (2pt) node[below]{\(P_3\)};
				\filldraw[black] (3,0) circle (2pt) node[below]{\(P_4\)};
				\filldraw[black] (0,3) circle (2pt) node[below]{\(P_5\)};
				\filldraw[black] (1,3) circle (2pt) node[below]{\(P_6\)};
				\filldraw[black] (2,3) circle (2pt) node[below]{\(P_7\)};
				\filldraw[black] (3,2) circle (2pt) node[below right]{\(P_8\)};
				\filldraw[black] (1,1) circle (2pt) node[below left]{\(P_9\)};
				\filldraw[black] (2,1) circle (2pt) node[below right]{\(P_{10}\)};
			\end{tikzpicture}
		\end{tabular}
	\end{center}

	The picture shows \(\alpha(X)\leq 4\), since the product of two linear forms and a conic belongs to \(I(X)\). One can use Macaulay 2 (\cite{GrSt}) to easily verify that \(\alpha(X) = 4\). From the picture it is immediate that \(\hyp(X) = 4\), so \(d(X)_1 = n - \hyp(X) = 10-4 = 6\). In this case,
	\[
		(k-1)(\alpha(X)-1-a)+1 = (3-1)(4-1-1)+1 = 5.
	\]
So \(d(X)_1 > (k-1)(\alpha(X)-1-a)+1\) for this set \(X\) and \(a = 1\).
\end{ex}

	Using more commutative algebra, we can obtain the following lower bound.

\begin{thm}\label{thm_better}
	Let \(X\subset{\mathbb P}^{k-1}, k\geq 3\), be a set of \(n\) points, not all contained in a hyperplane. Let \(1\leq a\leq \alpha(X)-1\) be an integer. Then,
	\[
		d(X)_a\geq {{\alpha(X)-1-a+k-1}\choose{k-1}}.
	\]
\end{thm}

\begin{proof}
	Define \(Y\) and \(m = |Y| = d(X)_a\) as in the proof of Proposition \ref{prop_alpha}. Recall \(R = {\mathbb K}[x_1,\dots,x_k]\). For convenience, label the proposed lower bound as \(\beta_a = \displaystyle {{\alpha(X)-1-a+k-1}\choose{k-1}}\). Then we want to show \(m \geq \beta_a\).
	
	By Lemma \ref{lem_reminders}(i, ii) we know \(\HF(R/I(Y),i)\) is nondecreasing in \(i\) and it is equal to \(m\) for \(i >> 0\). Thus it suffices to identify an integer \(i\) for which \(\HF(R/I(Y),i)\geq \beta_a\). The graded short exact sequence
	\[
		0 \to I(Y) \to R \to R/I(Y) \to 0
	\]
gives us
	\[
		\HF(R/I(Y),i) = \HF(R,i) - \HF(I(Y),i).
	\]
In particular, at \(i = \alpha(Y)-1\) we see \(HF(I(Y),\alpha(Y)-1) = 0\), so that
	\[
		\HF(R/I(Y),\alpha(Y)-1) = \HF(R,\alpha(Y)-1) = {{\alpha(Y)-1+k-1}\choose{k-1}}.
	\]
Since \(a\geq \alpha(X\setminus Y)\) (see the proof of Proposition \ref{prop_alpha} where \(X\setminus Y = X'\)), Lemma \ref{lem_reminders}(iii) tells us \(\alpha(Y)+a\geq \alpha(X)\), i.e., \(\alpha(Y)\geq \alpha(X)-a\). Therefore,
	\[
		\HF(R/I(Y),\alpha(Y)-1) = {{\alpha(Y)-1+k-1}\choose{k-1}} \geq {{\alpha(X)-1-a+k-1}\choose{k-1}} = \beta_a.
	\]
This completes the proof.
\end{proof}

\begin{prop}\label{prop_bound_comparison}
	Let \(X\subset{\mathbb P}^{k-1}, k\geq 3\), be a set of \(n\) points, not all contained in a hyperplane. Note \(\alpha(X)\geq 2\). Let \(1\leq a\leq \alpha(X)-1\) be an integer. Let \(\displaystyle \beta_a:= {{\alpha(X)-1-a+k-1}\choose{k-1}}\) be the lower bound from Theorem \ref{thm_better} and let \(\beta_a':= (k-1)(\alpha(X)-1-a)+1\) be the lower bound from Corollary \ref{cor_alphabound}. Then \(\beta_a\) is a better bound than \(\beta_a'\). That is, \(\beta_a = \beta_a'\) when \(\alpha(X)\in \{a+1,a+2\}\) and we have \(\beta_a > \beta_a'\) when \(\alpha(X)\geq a+3\).
\end{prop}

\begin{proof}
	When \(\alpha(X) = a+1\) and when \(\alpha(X) = a+2\), the formulas give \(\beta_a = 1 = \beta_a'\) and \(\beta_a = k = \beta_a'\), respectively. Now suppose \(\alpha(X)\geq a+3\).
	
	Vandermonde's Identity (or Convolution) is the following equality which is true for all nonnegative integers \(A\), \(B\), and \(C\):
	\[
		{A+B\choose C} = \sum_{j=0}^C {A\choose j}{B\choose C-j}.
	\]
The integers \(B = \alpha(X) - 1 - a\) and \(A = C = k-1\) are nonnegative when \(\alpha(X) \geq a + 3\) and \(k\geq 3\), and Vandermonde's Identity tells us
	\[
		\beta_a = {\alpha(X)-1-a+k-1 \choose k-1} = \sum_{j=0}^{k-1} {k-1\choose j}{\alpha(X) - 1 - a\choose k-1-j}.
	\]
Since \(k\geq 3\), the sum on the right has at least three terms. The last two terms, where \(j = k-2\) and \(j = k-1\), sum to \(\beta_a'\). However, the third-to-last term (\(j = k-3\)) is \(\displaystyle {k-1\choose 2}{\alpha(X) - 1 - a \choose2} > 0\) when \(\alpha(X)\geq a+3\) and \(k\geq 3\). Therefore \(\beta_a > \beta_a'\) in this case because any remaining terms in the sum are nonnegative.
\end{proof}

	Of course, for \(a \geq 2\) it may be possible that better bounds than \(\beta_a\) can be found. Room for improvement exists in the proof of Theorem \ref{thm_better}: In that proof, we obtained our bound \(\beta_a\) by using \(\alpha(Y)\geq \alpha(X) - a\). If we could say more about \(\alpha(Y)\), we might have a better bound.
	
	For most of our examples we work with \(a=1\) because it is simpler geometrically: If we would choose \(a\geq 2\), then the hypersurface of degree \(a\) that contains the maximum \(\hyp(X)_a\) points of \(X\) may have irreducible components of degrees 2 or greater, and this will cause difficulties in choosing the relevant examples of \(X\).
	
\begin{ex}\label{ex_boundcomp2}
	Let \(X\) be the set from Example \ref{ex_boundcomp1}. In that example we showed \(d(X)_1 = 6\), \(\alpha(X) = 4\), and \(\beta_1' = 5\). We had \(\alpha(X) = a+3\), so by Proposition \ref{prop_bound_comparison} it should be the case that \(\beta_1\) is a better lower bound than \(\beta_1'\). Indeed, \(\beta_1 = d(X)_1 = 6 > \beta_1'\) for this set \(X\).
	
	Now, if we take \(a=2=\alpha(X)-2\), then, looking at \(P_1,\ldots, P_7\) as being on the conic that is the union of the corresponding two lines, \(\hyp(X)_2=7\) (recall Remark \ref{rem_evalcodes} (1)). Therefore \(d(X)_2=10-7=3\), which equals \(\beta_2=\beta_2'=3\).

	Lastly, taking \(a=3=\alpha(X)-1\), and looking at \(P_1,\ldots, P_9\), then \(\hyp(X)_3=9\), giving that \(d(X)_3=10-9=1\), which also equals \(\beta_3=\beta_3'\).
\end{ex}

\begin{ex}\label{ex_k_points_bounds_check}
	In this and the next example we are interested into providing evidence that it is possible to have \(d(X)_1=\beta_1=\beta_1'\).
	
	In \({\mathbb P}^{k-1}\) let \(X\) be a set of \(k\) points not contained in a hyperplane. In particular, \(\alpha(X)\geq 2\). Moreover, choose any \(k-1\) points of \(X\); they belong to some hyperplane \(V(L_1)\), and the remaining point belongs to another hyperplane \(V(L_2)\). Thus \(L_1L_2\in I(X)\) and \(\alpha(X) \leq \deg(L_1L_2) = 2\). So \(\alpha(X) = 2\). In this case \(\beta_1 = \displaystyle {{2+k-3}\choose{k-1}} = 1\), and since \(\alpha(X)=a+1\) we therefore have \(\beta_1=\beta_1'\). Since \(\hyp(X) = k-1\) and \(|X| = n = k\), by Proposition \ref{prop_mindistpoints} we have \(d(X)_1 = n - \hyp(X) = 1 = \beta_1=\beta_1'\).
\end{ex}

\begin{ex}\label{ex_seven_points}
	Here is a more concrete example, and one for which \(d(X)_1 > 1\). Consider the set of \(7\) points \(X= \{P_1,\dots,P_7\}\subset {\mathbb P}^2\) whose coordinates make up the columns of the matrix \(G\) below.

	\begin{center}
		\begin{tabular}{@{}m{\dimexpr.5\textwidth-.5\columnsep}m{\dimexpr.5\textwidth-.5\columnsep}@{}}
			\( \displaystyle
				G =	\begin{pmatrix}
						0 & 1 & 2 & 3 & 0 & 0 & 1\\
						0 & 0 & 0 & 0 & 1 & 2 & 1\\
						1 & 1 & 1 & 1 & 1 & 1 & 1
					\end{pmatrix}
			\)
			&
			\begin{tikzpicture}[scale=1]
				\filldraw[black] (0,0) circle (2pt) node[left]{\(P_1\)};
				\filldraw[black] (1,0) circle (2pt) node[left]{\(P_2\)};
				\filldraw[black] (2,0) circle (2pt) node[left]{\(P_3\)};
				\filldraw[black] (3,0) circle (2pt) node[left]{\(P_4\)};
				\filldraw[black] (0,1) circle (2pt) node[left]{\(P_5\)};
				\filldraw[black] (0,2) circle (2pt) node[left]{\(P_6\)};
				\filldraw[black] (1,1) circle (2pt) node[left]{\(P_7\)};
			\end{tikzpicture}
		\end{tabular}
	\end{center}

	From the picture it is clear that \(\hyp(X) = 4\), thus, from Proposition \ref{prop_mindistpoints}, we then have \(d(X)_1 = |X| - \hyp(X) = 3\). Notice that, by drawing two more lines through the three points not on \(V(y)\), we see \(xy(x-z)\in I(X)\), thus \(\alpha(X)\leq 3\). Using Macaulay 2 (\cite{GrSt}) we confirm that \(\alpha(X) = 3\). One can then quickly see \(d(X)_1 = \beta_1 = \beta_1' = 3\) using the provided formulas.
\end{ex}

The previous example can be generalized to any \(k\geq 3\) as in the following proposition.

\begin{prop}
	For any \(k\geq 3\) and sufficiently large base field \({\mathbb K}\), there exists a set \(X\subset {\mathbb P}^{k-1}\) of \(n = 3k-2\) points such that \(\alpha(X) = 3\) and \(d(X)_1 = \beta_1 = \beta_1' = k\).
\end{prop}

\begin{proof}
	Construct a set \(X\subseteq {\mathbb P}^{k-1}\) of \(n = 3k-2\) points as follows. Since \({\mathbb K}\) is sufficiently large, we may choose \(2k-2\) points in a hyperplane \(V(L)\). Then, choose \(k\) points which do not belong to \(V(L)\). Note \(X\) is not contained in a single hyperplane, so \(\alpha(X)\geq 2\). By choosing the \(k\) points off \(V(L)\) generically enough, we may assume \(\alpha(X) \geq 3\).
	
	By construction, \(\hyp(X)\geq 2k-2\). Suppose for a contradiction that \(2k-1\) points of \(X\) lie on some hyperplane \(V(L')\). Then at least \((2k-1) - k = k-1\) of those chosen points belong to \(V(L)\), meaning \(V(L) = V(L')\). This is a contradiction to the construction of \(X\) because this means more than \(2k-2\) points of \(X\) belong to \(V(L)\). Thus \(\hyp(X) = 2k-2\). Additionally, the \(k\) points of \(X\setminus V(L)\) belong to the union of 2 hyperplanes (as any set of \(k-1\) points belongs to a unique hyperplane) which we can identify as \(V(f_1)\) and \(V(f_2)\). Then \(Lf_1f_2\in I(X)\), so \(\alpha(X)\leq 3\) which implies \(\alpha(X) = 3\).
	
	Now, by Proposition \ref{prop_mindistpoints}, we have \(d(X)_1 = |X|-\hyp(X) = (3k-2)-(2k-2) = k\). With \(a = 1\) and \(\alpha(X) = 3\), we also have \(\beta_1 = \beta_1' = k\).
\end{proof}


\section{Minimum distance and the minimum socle degree}\label{sec_mindistsocle}

	In this section we will generalize \cite[Theorem 1]{BaFo}, which says the following: let \(X\subset{\mathbb P}^{k-1}\) be a finite set of points. Suppose \(X\) is a complete intersection of hypersurfaces of degrees \(e_1,\dots,e_{k-1}\) where we index the degrees so that \(1\leq e_1\leq\cdots\leq e_{k-1}\). Suppose further that \(X\) is in general linear position (i.e., any \(u\leq k\) points of \(X\) will span a \({\mathbb P}^{u-1}\); see Definition \ref{defn_generic} (2)). Then, for any \(1\leq a\leq \reg(X)-1\), one has
	\[
		d(X)_a\geq (k-1)(\reg(X)-1-a)+2.
	\]
Because \(X\) is a complete intersection, its Castelnuovo-Mumford regularity is \(\reg(X)=e_1+\cdots+e_{k-1}-(k-1)\). By using \cite{EiPo}, one can replace the ``complete intersection'' condition with the more general condition ``Gorenstein''.

\begin{rem}\label{rem_soc}
	We mentioned before that for a finite set of points \(X\subset{\mathbb P}^{k-1}\), we have
	\[
		\reg(X)\geq {\bf v}(X)\geq s(X)\geq \alpha(X)-1.
	\]
	\begin{itemize}
		\item[(i)] If \(X\) is as in Remark \ref{rem_generic} (iii), since \(\alpha(X)=\reg(X)+1=s(X)+1\), then the bound in Corollary \ref{cor_alphabound} becomes \(d(X)_a\geq (k-1)(\reg(X)-a)+1=(k-1)(s(X)-a)+1\).

		\item[(ii)] Let \(X\subset{\mathbb P}^{k-1}\) be a finite set of \(n\) points in generic position (see Definition \ref{defn_generic} (1)), and suppose \(\displaystyle {{m-1+k-1}\choose{k-1}}< n < {{m+k-1}\choose{k-1}}\), for some positive integer \(m\). Then, by \cite[Theorem 2.2]{Lo1} (or the earlier work \cite[Corollary 1.6]{GeMa}), we have \(\reg(X)=m\) and \(\alpha(X)=m\). With this, the bound in Corollary \ref{cor_alphabound} becomes \(d(X)_a\geq (k-1)(\reg(X)-1-a)+1\). In this case \(s(X)\) can only be \(m\) or \(m-1\), so, in general \(d(X)_a\geq (k-1)(s(X)-1-a)+1\). Note, the case where \(n\) is equal to the lower bound is considered in Remark \ref{rem_generic} (iv).

		\item[(iii)] Since in general \(\reg(X)\geq \alpha(X)-1\), the bound obtained in Corollary \ref{cor_alphabound} and the Ballico-Fontanari bound are comparable, with the latter being the better one (except for the cases (i) and (ii)). But, this second bound happens under the restrictive conditions ``complete intersection'' and ``general linear position'', whereas the bound in the corollary is true for any set \(X\). Since we still want to use the shifts in the ``tail'' of the free resolution to obtain a similar bound, we will drop the ``complete intersection'' condition, but keep the ``general linear position condition'' to show that
			\[
				d(X)_a\geq (k-1)(s(X)-1-a)+2.
			\]
		
		\item[(iv)] As we mentioned already in Remark \ref{rem_generic} (i), in coding theory lingo, the ``general linear position'' condition translates to the code defined by \(X\) being MDS. \cite{Ha1} emphasizes very well that when \(a\geq 2\), the value of \(d(X)_a\) for such sets of points is very challenging to bound, unless some extra algebraic geometric conditions such as ``complete intersection'', or ``level'' (i.e., \(\reg(X)=s(X)\)) are added.

	\end{itemize}
\end{rem}

	Let \(Y\subset{\mathbb P}^{k-1}\) be a finite set of points. We define the {\em rank of \(Y\)} to be the rank of the matrix having as columns the coordinates of the points of \(Y\); we denote this number \(\rk(Y)\). In other words, \({\mathbb P}^{\operatorname{rk}(Y)-1}\) is the smallest projective space which \(Y\) can be embedded into via a change of variables.

\begin{thm}\label{thm_socle}
	Let \(X\subset{\mathbb P}^{k-1}\) be a finite set of \(n\geq k\geq 3\) points in general linear position. Let \(1\leq a\leq s(X)-1\) be an integer. Then, we have either
		\[
			d(X)_a\leq k-1 \ \ \ \text{or}\ \ \ d(X)_a\geq (k-1)(s(X)-1-a)+2.
		\]
\end{thm}

\begin{proof}
	Denote \(X=\{P_1,\ldots,P_n\}\). Let \(m:=d(X)_a\geq 1\). By Proposition \ref{prop_paramEvaluation}, let \(X'\subsetneq X\) be such that \(|X'|=n-m\), let \(Y:=X\setminus X'=\{Q_1,\ldots,Q_m\}\), and let \(f\in I(X')\) with \(\deg(f)=a\) and \(f(Q_j)\neq 0\) for all \(j=1,\ldots,m\) (from the maximality of \(|X'|\)).

	If \(\rk(Y)\leq k-1\), because \(Y\subset X\) and \(X\) is in general linear position, then \(|Y|=\rk(Y)\), and so \(d(X)_a = |Y| \leq k-1\) as claimed.
	
	Suppose instead \(d(X)_a = |Y|\geq k\). In several steps we will show \(d(X)_a\geq (k-1)(s(X)-1-a)+2\).
	
	\medskip

	\underline{Step 1:} First, we show \(s(Y)\geq s(X)-a\).
	
	Let \(L\in R:={\mathbb K}[x_1,\ldots,x_k]\) be a linear form such that \(L(P_i)\neq 0\) for all \(i=1,\ldots,n\). That is, \(L\) is a non-zerodivisor on \(R/I(X)\). Then, \(L\) is a non-zerodivisor on \(R/I(Y)\), as \(Y\subset X\).

	Let \(g+\langle I(Y),L\rangle \in {\rm Soc}(R/\langle I(Y),L\rangle)\), with \(\deg(g)=s(Y) = \alpha({\rm Soc}(R/\langle I(X),L\rangle))\) and \(g\notin I(Y)\). Then, \(g(Q)\neq 0\) for some \(Q\in Y\), and, for \(l=1,\ldots,k\), we have
	\[
		x_lg=h_l+Lg_l, \mbox{ for some } h_l\in I(Y), g_l\in R.
	\]
Multiplying the above by \(f\), for \(l=1,\ldots,k\) we have that \(x_l(fg)=(fh_l)+L(fg_l)\). But \((fg)(Q)\neq 0\), and \(fh_l\in I(X)\). Therefore, \(fg+\langle I(X),L\rangle\in {\rm Soc}(R/\langle I(X),L\rangle)\), and since \(\deg(fg)=a+s(Y)\), we obtain
	\[
		s(Y)\geq s(X)-a.
	\]
	
	\medskip

	\underline{Step 2:} Now we show if \(Q\in Y\) and \(h\in I(Y\setminus\{Q\})\setminus I(Y)\), meaning \(h\in I(Y\setminus\{Q\})\) and \(h(Q)\neq 0\), then \(\deg(h)\geq s(Y)\).
	
	Suppose \(Q=[1,0,\ldots,0]\), and let \(\tilde{L}=x_1+c_2x_2+\cdots+c_kx_k\in R_1\) be a linear form such that \(\tilde{L}(Q_j)\neq 0\) for all \(j=1,\ldots,m\), i.e., \(\tilde{L}\) is a non-zerodivisor on \(R/I(Y)\).

	Let \(h\in I(Y\setminus\{Q\})\setminus I(Y)\). Then \((x_1-\tilde{L})h\) vanishes at all points of \(Y\), so \(x_1h=\tilde{L}h+h_1\), where \(h_1\in I(Y)\). Also, for all \(l=2,\ldots,k\), we see \(x_lh\) vanishes at all points of \(Y\), so \(x_lh\in I(Y)\). Therefore, \(h+\langle I(Y),\tilde{L}\rangle \in {\rm Soc}(R/\langle I(Y),\tilde{L}\rangle)\), hence \(\deg(h)\geq s(Y)\).

	\medskip

	\underline{Step 3:} Now we use the assumption that \(d(X)_a = |Y| \geq k\). Write \(|Y|=d(X)_a=m=e(k-1)+\ell\), where \(\ell\in\{0,\ldots,k-2\}\). We have these cases:

	\begin{itemize}
		\item[(i)] If \(\ell\geq 2\), choose a point \(Q\in Y\). As \(Y\) is in general linear position (being a subset of \(X\)), the remaining points of \(Y\setminus\{Q\}\) can be placed on \(e+1\) hyperplanes such that none of the hyperplanes will pass through \(Q\): the \(e(k-1)\) points sit on exactly \(e\) hyperplanes, and the remaining \(\ell-1\leq k-3\) can be placed on a hyperplane not passing through \(Q\). From Step 2 we have \(e+1\geq s(Y)\), and so, \(e(k-1)+\ell\geq (k-1)(s(Y)-1)+\ell\). Therefore, since \(\ell\geq 2\),
			\[
				m\geq (k-1)(s(Y)-1)+2.
			\]
		\item[(ii)] If \(\ell=1\), choose a point \(Q\in Y\). The remaining \(e(k-1)\) points of \(Y\setminus\{Q\}\) can be placed on \(e\) hyperplanes, none of which pass through \(Q\), and so, from Step 2 we have \(e\geq s(Y)\). Therefore,
			\[
				m\geq (k-1)s(Y)+1.
			\]
		\item[(iii)] If \(\ell=0\), again we choose \(Q\in Y\), and place the remaining \(e(k-1)-1=(e-1)(k-1)+k-2\) points on \(e=e-1+1\) hyperplanes, none passing through \(Q\): similar to case (i), the \((e-1)(k-1)\) points sit on exactly \(e-1\) hyperplanes, and the remaining \(k-2\) can be placed on a hyperplane not passing through \(Q\). So, in this case, we obtain
			\[
				m\geq (k-1)s(Y).
			\]
	\end{itemize}

	When \(k\geq 3\), the three cases above each imply
	\[
		d(X)_a = m\geq (k-1)(s(Y)-1)+2.
	\]
Do note that the inequalities from cases (ii) and (iii) are stronger than the inequality from (i) thus, if we somehow know \(\ell \equiv d(X)_a \mod (k-1)\), we get a better conclusion.
\end{proof}

\begin{ex}
	Consider the set of \(10\) points \(X= \{P_1,\dots,P_{10}\}\subset {\mathbb P}^3\) whose coordinates make up the columns of the matrix \(G\) below. When the field \({\mathbb K}\) has characteristic 0 or for sufficiently large positive characteristic, the set \(X\) is in general linear position, which can be verified by checking that each set of 4 columns of \(G\) is linearly independent.\footnote{We would like to thank the anonymous referee who pointed out to us that the smallest positive characteristic for which this example works is \(p = 103\).}
	\[
		G =	\begin{pmatrix}
				8 & 4 & 9 & 8 & 6 & 2 & 0 & 0 & 3 & 1\\
				4 & 5 & 3 & 7 & 0 & 8 & 5 & 2 & 5 & 0\\
				3 & 0 & 2 & 0 & 2 & 4 & 2 & 7 & 0 & 9\\
				0 & 6 & 8 & 2 & 4 & 1 & 0 & 1 & 3 & 2
			\end{pmatrix}
	\]
	By Remark \ref{rem_generic} (i), \(d(X)_1 = 10-4+1=7\). At the same time, using Macaulay 2 (\cite{GrSt}), the graded minimal free resolution of \(R/I(X)\), where \(R = {\mathbb K}[x_1,\dots,x_4]\), is
	\[
		0\to R(-5)^6\to R(-4)^{15}\to R(-3)^{10}\to R \to R/I(X)\to 0.
	\]

	From the free resolution, we observe \(s(X) = 5-3=2\), and \(X\) is not a complete intersection. We see that Theorem \ref{thm_socle} is true for this example:
	\[
		7 \geq (4-1)(2-1-1) +2=2.
	\]

	The gap between \(d(X)_1=7\) and the lower bound \(2\) is big, but as we will see later in Example \ref{ex_socle_bound_attained}, this lower bound can be attained.

	Also observe that since \(s(X) = \reg(X) = 2\), then, by Remark \ref{rem_vnumber} (i), \(d(X)_a=1\) for all \(a\geq 2\).
\end{ex}


\section{Discussion and additional examples}\label{sec_discussionExamples}

	After reading the above results, it is natural to ask if there are any ``nice'' or somehow ``minimal'' examples of sets \(X\subseteq {\mathbb P}^{k-1}\) with the properties from the theorems, lemmas, and so on.

	Let us focus on sets of points for which the bound in Theorem \ref{thm_better} is attained. So, let \(X\subset{\mathbb P}^{k-1}, k\geq 3\), be a set of \(n\) points, not all contained in a hyperplane. Let \(1\leq a\leq \alpha(X)-1\) be an integer, and suppose \(\displaystyle d(X)_a=\beta_a:= {{\alpha(X)-1-a+k-1}\choose{k-1}}\).

	Define \(Y\) as in the proof of Proposition \ref{prop_alpha}, so \(|Y|=\beta_a\). Recall \(R = {\mathbb K}[x_1,\dots,x_k]\). For convenience, denote \(\alpha:=\alpha(X)\), \(\alpha':=\alpha(Y)\), \(r:=\reg(X)\), and \(r':=\reg(Y)\). With this we have \(\alpha\leq \alpha'+a\).

	If we compute the Hilbert function values we have:
	\[
		\begin{array}{r|ccccc}
			i & 0 & \cdots & \alpha'-1 & \alpha' & \cdots \\ \hline
			\HF(R/I(Y),i) & 1 & \cdots & {\alpha'-1+k-1\choose k-1} & * & \cdots
		\end{array}
	\]

	Since \(\alpha'-1\geq \alpha-a-1\), from Lemma \ref{lem_reminders} (i) and (ii) we have the inequalities
	\[
		{\alpha-a-1+k-1\choose k-1}\leq {\alpha'-1+k-1\choose k-1} \leq * \leq |Y|=\beta_a = {\alpha-a-1+k-1 \choose k-1},
	\]
which of course means we have equalities throughout. This gives the following properties of \(Y\):

\begin{itemize}
	\item[(a)] \(Y\) is in generic position.
	\item[(b)] \(r'=\alpha'-1\).
	\item[(c)] \(\alpha'=\alpha-a\).
\end{itemize}

	In Examples \ref{ex_boundcomp1} and \ref{ex_boundcomp2} we demonstrated that \(\beta_1\) is a better bound than \(\beta_1'\), and while doing that we obtained that \(d(X)_1=\beta_1\). Indeed, in Example \ref{ex_boundcomp1} the set \(Y\) of the six points not on the line \(V(y)\) is in generic position (but because \(P_5,P_6,P_7\in V(y-3z)\), the set \(Y\) is not in general linear position). From Remark \ref{rem_generic} (iii), \(\reg(Y)=2\) and \(\alpha(Y)=3\), confirming that \(\alpha(X)=4=\alpha(Y)+1\) (since \(a=1\)).

	So, in order to construct examples of sets \(X\) for which \(d(X)_a=\beta_a\), we will do the following:

\begin{itemize}
	  \item First, for some integer \(m\geq a+3\) we choose a set \(Y\subset{\mathbb P}^{k-1}\) of \(\displaystyle {m-1-a+k-1\choose k-1}\) in generic position. Often, randomly generated points will be in generic position, so it is easy to obtain these points for constructing examples.
	  \item Next, we choose a set \(X'\) of sufficiently many points outside \(Y\) on a degree \(a\) hypersurface.
	  \item Taking the disjoint union \(X=Y\cup X'\), if \(Y\) is chosen generically enough and if the cardinality of \(X'\) is sufficiently large, we should have \(d(X)_a=|Y|\) and \(\alpha(X)=m\), hence \(X\) satisfies the bound in Theorem \ref{thm_better}, and \(d(X)_a=\beta_a>\beta_a'\).
\end{itemize}

	Of course, being able to choose these points depends heavily on the size and characteristic of the field \({\mathbb K}\) being sufficiently large.

\begin{rem}\label{rem_example_construction}
	In this remark we present more details about the above construction for \(a=1\).
	
	Choose \(k\geq 3\). Fix an integer \(m\geq 4\). Choose a set of \(\displaystyle {m+k-3\choose k-1}\) points \(Y\) in generic position and identify a hyperplane \(V(L)\) which does not intersect \(Y\). Choose some number of points \(\{P_1,\dots,P_\ell\}\), where \(\ell\geq 1\), in \(V(L)\) and set \(X = Y\cup\{P_1,\dots,P_\ell\}\). Our goal now is to show that by choosing \(\ell\) sufficiently large, we get \(\alpha(X) = m\), \(\hyp(X) = \ell\), and \(\displaystyle d(X)_1 = |Y| = {m+k-3\choose k-1}\). The last condition is immediate if \(\hyp(X) = \ell\), and the condition \(\hyp(X) = \ell\) is clearly guaranteed by choosing \(\ell\) sufficiently large. We only need to show that for \(\ell\) large enough, \(\alpha(X) = m\), too.
	
	Since \(Y\) is in generic position, by Remark \ref{rem_generic}(iii) we know \(\alpha(Y) = (m-2)+1=m-1\) and \(\mu:=\mu(I(Y))=\displaystyle {m+k-3\choose k-2}\). By Lemma \ref{lem_reminders}(iii), we know \(\alpha(Y)\leq \alpha(X)\leq \alpha(Y)+1\), so \(m-1\leq \alpha(X)\leq m\).

	In particular, \(I(Y)\) is generated by \(\mu\) degree \(m-1\) homogeneous polynomials; call them \(f_1,\dots,f_{\mu}\). If \(\alpha(X) = m-1\), then there is a nonzero degree \(m-1\) polynomial \(g\in I(X)\subset I(Y)\). Therefore there are \(\mu\) constants \(a_1,\dots,a_{\mu}\in {\mathbb K}\) such that
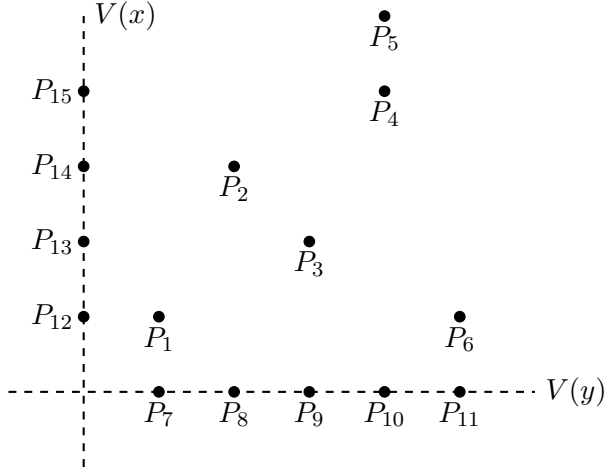
	\[
		g = a_1f_1+\cdots+a_{\mu}f_{\mu}.
	\]

	We have \(g(P_i)=0\), for all \(i=1,\ldots,\ell\), so by evaluating the above at the points \(P_i\) we obtain a homogeneous linear system of \(\ell\) equations in variables \(a_1,\ldots,a_{\mu}\), with an \(\ell\times \mu\) coefficients matrix \(M:=(f_j(P_i))_{i,j}\). Suppose \(\ell\geq\mu\). By choosing \(Y\) generically enough, we can assume that the coefficients of the \(f_j\) result in the matrix \(M\) having rank \(\mu\). This gives that \(a_1=\cdots=a_{\mu}=0\), a contradiction. Therefore, choosing \(\ell\geq\mu\) will produce \(\alpha(X)=m\).
	
	In Example \ref{ex_boundcomp1} we had \(m=4\) and \(k=3\) which gave \(\operatorname{hyp}(X)=\ell=\mu=m=4\).
\end{rem}

\begin{ex}\label{ex_aTwo}
	In this example we look a bit at the case when \(a=2\). For simplicity, we are still working in \({\mathbb P}^2\), so \(k=3\), and we will set \(m=5\).

	In this case, \(Y\) is a set of \(\displaystyle {5-1-2+2\choose 2}=6\) points in \({\mathbb P}^2\). Again, we have \(\alpha(Y)=3\), and \(\mu(I(Y))=4\); i.e. \(I(Y)=\langle f_1,\ldots,f_4\rangle\subset R:={\mathbb K}[x_1,x_2,x_3]\), where \(\deg(f_i)=3\), for \(i=1,\ldots,4\).

	We have \(\alpha(X)\leq \alpha(Y)+a=5\). We want to show that if we choose \(\ell\) points on a curve of degree \(a=2\), with \(\ell\) sufficiently large, we have equality and \(|X| - d(X)_a =\ell\).
	
	We can argue similarly as in Remark \ref{rem_example_construction}. Suppose to the contrary that \(\alpha(X)<5\), so there exists a nonzero polynomial \(g\) of degree 4 in \(I(X)\). Then, since \(I(X)\subset I(Y)\), we have
	\begin{align*}
		g	&=\underbrace{(a_1x_1+b_1x_2+c_1x_3)}_{L_1}f_1+\underbrace{(a_2x_1+b_2x_2+c_2x_3)}_{L_2}f_2\\
			&\hspace{6em} +\underbrace{(a_3x_1+b_3x_2+c_3x_3)}_{L_3}f_3+ \underbrace{(a_4x_1+b_4x_2+c_4x_3)}_{L_4}f_4
	\end{align*}
where \(a_1,\ldots,a_4,b_1,\ldots,b_4,c_1\ldots,c_4\in{\mathbb K}\).

	Choosing \(\ell\) to be larger than 12, which is the number of these constants, by choosing \(Y\) carefully we can obtain the contradiction \(g=0\).
	
	Working with \(\ell > 12\) guarantees an example with the above construction has the desired properties. However, it is not strictly necessary to have \(\ell > 12\) to have the desired properties. Here is a more specific example with \(\ell=9\). Consider the set of \(15\) points \(X= \{P_1,\dots,P_{15}\}\subset {\mathbb P}^2\) whose coordinates make up the columns of the matrix \(G\) below. The set \(Y\) will consist of the points represented by the first 6 columns while the remaining 9 points will lie on a degree 2 curve.
	\[
		G =	\begin{pmatrix}
				1 & 2 & 3 & 4 & 4 & 5 & 1 & 2 & 3 & 4 & 5 & 0 & 0 & 0 & 0\\
				1 & 3 & 2 & 4 & 5 & 1 & 0 & 0 & 0 & 0 & 0 & 1 & 2 & 3 & 4\\
				1 & 1 & 1 & 1 & 1 & 1 & 1 & 1 & 1 & 1 & 1 & 1 & 1 & 1 & 1
			\end{pmatrix}
	\]
	\begin{center}
		\begin{tikzpicture}
			\draw[thick, dashed] (-1,0) -- (6,0) node[right]{\(V(y)\)};
			\draw[thick, dashed] (0,-1) -- (0,5) node[right]{\(V(x)\)};
			\filldraw[black] (1,1) circle (2pt) node[below]{\(P_1\)};
			\filldraw[black] (2,3) circle (2pt) node[below]{\(P_2\)};
			\filldraw[black] (3,2) circle (2pt) node[below]{\(P_3\)};
			\filldraw[black] (4,4) circle (2pt) node[below]{\(P_4\)};
			\filldraw[black] (4,5) circle (2pt) node[below]{\(P_5\)};
			\filldraw[black] (5,1) circle (2pt) node[below]{\(P_6\)};
			\filldraw[black] (1,0) circle (2pt) node[below]{\(P_7\)};
			\filldraw[black] (2,0) circle (2pt) node[below]{\(P_8\)};
			\filldraw[black] (3,0) circle (2pt) node[below]{\(P_9\)};
			\filldraw[black] (4,0) circle (2pt) node[below]{\(P_{10}\)};
			\filldraw[black] (5,0) circle (2pt) node[below]{\(P_{11}\)};
			\filldraw[black] (0,1) circle (2pt) node[left]{\(P_{12}\)};
			\filldraw[black] (0,2) circle (2pt) node[left]{\(P_{13}\)};
			\filldraw[black] (0,3) circle (2pt) node[left]{\(P_{14}\)};
			\filldraw[black] (0,4) circle (2pt) node[left]{\(P_{15}\)};
		\end{tikzpicture}
	\end{center}
	
	Using Macaulay 2 (\cite{GrSt}), one can verify that for \(Y = \{P_1,\dots, P_6\}\), we have \(\alpha(Y) = 3\) as mandated by the construction at the start of this example, and \(I(Y)\) is generated by four cubics. We have \(\alpha(X) = 5\), the maximum possible value. The \(\ell = 9\) points \(P_7,\dots,P_{15}\) lie on the degree \(a=2\) curve \(V(xy)\). Indeed we have
	\[
		d(X)_2=6=\beta_2={{5-1-2+3-1}\choose{3-1}}.
	\]
	
	As a side note, using Macaulay 2 (\cite{GrSt}), we discover more interesting properties about the set \(X\). The Hilbert series for \(R/I(X)\) is \(\{1,3,6,10,15,15,15,\dots\}\). From Definition \ref{defn_generic} (1), the set \(X\) is in generic 15-position. By simply looking at the picture of \(X\), we see it is not in general linear position: we can easily identify a set of more than 2 collinear points. Furthermore, since \(|X|=15=\displaystyle {{5-1+3-1}\choose{3-1}}\), from Remark \ref{rem_generic} (iv) we know the regularity of \(X\) is \(\reg(X) = s(X) = \alpha(X)-1 = 4\). From the same remark we expect \(R/I(X)\) to have a linear graded free resolution, and in fact it has the following resolution.
	\[
		0 \to R(-6)^5 \to R(-5)^5 \to R \to R/I(X) \to 0.
	\]
\end{ex}

\begin{ex}\label{ex_socle_bound_attained}
	We will show the lower bound from Theorem \ref{thm_socle} can be attained. In \({\mathbb P}^2\) let \(Y\) be the complete intersection of an irreducible quintic and an irreducible conic, so \(|Y| = 10\). (If necessary, take an extension field large enough that the points \(Y\) are all rational over the base field.)\footnote{We would like to thank an anonymous referee for their suggestion that helped us generalize this example to fields which are not algebraically closed.} Let \(X\) be the union of \(Y\) with the four points in \(V(x_1(x_1-x_3),x_2(x_2-x_3))\). Note, then, that \(X\) is a set of \({\mathbb K}\)-rational points.
	
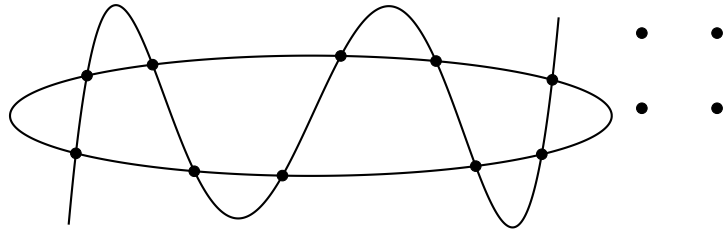
\begin{center}
		\begin{tikzpicture}[scale=2]
			\draw[smooth,thick,black,domain=0.39:3.65,samples=400]
				plot (\x,{.2+(\x-.5)*(\x-1)*(\x-2.1)*(\x-2.9)*(\x-3.6)});
			
			\draw[smooth,thick,black,variable=\t,domain=0:2*pi,samples=400]
				plot ({2*cos(\t r)+2},{.4*sin(\t r)});
	
			\filldraw[black] (0.512,0.267) circle (1pt) ;
			\filldraw[black] (0.948,0.34) circle (1pt) ;
			\filldraw[black] (2.199,0.398) circle (1pt) ;
			\filldraw[black] (2.832,0.364) circle (1pt) ;
			\filldraw[black] (3.605,0.239) circle (1pt) ;
			\filldraw[black] (3.535,-0.256) circle (1pt) ;
			\filldraw[black] (3.096,-0.335) circle (1pt) ;
			\filldraw[black] (1.811,-0.398) circle (1pt) ;
			\filldraw[black] (1.225,-0.369) circle (1pt) ;
			\filldraw[black] (0.438,-0.25) circle (1pt) ;
			
			\filldraw[black] (4.2,0.55) circle (1pt) ;
			\filldraw[black] (4.7,0.55) circle (1pt) ;
			\filldraw[black] (4.2,0.05) circle (1pt) ;
			\filldraw[black] (4.7,0.05) circle (1pt) ;
		\end{tikzpicture}
	\end{center}
	Recall from Remark \ref{rem_evalcodes} (1), \(\hyp(X)_a\) denotes the maximum number of points in a proper subset of \(X\) which lie on a degree \(a\) hypersurface. Then \(d(X)_a = n - \hyp(X)_a\).
	
	By the choosing of the quintic and the conic generally enough, \(X\) can be constructed to be in general linear position. From the picture, observe that \(\hyp(X)_2 = |Y| = 10\), so \(d(X)_2 = 14-10=4 > 2 = k -1\).

	With Macaulay 2 (\cite{GrSt}), we obtain \(s(X) = 4\), and therefore, the bound in Theorem \ref{thm_socle} is attained:
	\[
		(k-1)(s(X)-1-a)+2  = (3-1)(4-1-2)+2 = 4 = d(X)_2.
	\]
	Observe in this example that \(\reg(X)=5\), and hence the lower bound we are discussing here will not be true if one is tempted to replace \(s(X)\) with \(\reg(X)\).
	
	Since \(X\) is in general linear position, by taking \(Y\), which is contained in the conic, and two of the other points, which are contained in a line, we have \(\hyp(X)_3 = 12\) so \(d(X)_3 = 2\). We have \(d(X)_3 \leq 2 = k-1\). However, note that
	\[
		(k-1)(s(X)-1-a)+2 = (3-1)(4-1-3)+2 = 2 = d(X)_3.
	\]

	In the proof of Theorem \ref{thm_socle}, for \(1\leq a\leq s(X)-1\), the lower bound is only proved to be valid when \(d(X)_a > k-1\). This example begs the question: When do we have \((k-1)(s(X)-1-a)+2\leq d(X)_a\leq k-1\), i.e., when is the lower bound valid despite \(d(X)_a\) not being ``large enough''?
\end{ex}

\begin{rem}
	Let \(X\subset {\mathbb P}^{k-1}\), where \(k\geq 3\), be a set of points in general linear position.  Suppose for an integer \(a\), where \(1\leq a\leq s(X)-1\), that
	\[
		(k-1)(s(X)-1-a)+2 \leq d(X)_a \leq k-1.
	\]
A quick check will show that this inequality necessitates \(a = s(X)-1\), so the above inequality becomes
	\[
		2 \leq d(X)_{s(X)-1}\leq k-1.
	\]
Thus, the lower bound is valid when \(d(X)_a \leq k-1\) only for one specific value of \(a\) (the largest possible value we considered in the theorem) and the lower bound can only ever equal \(k-1\) when \(k = 3\). Notice that at the end of Example \ref{ex_socle_bound_attained}, we looked at \(d(X)_a\) where \(a = s(X)-1\) and \(k = 3\).
\end{rem}


\vskip 0.2in

\noindent \textbf{Acknowledgments}

We are very grateful to the anonymous referees for their comments and suggestions which improved the clarity of this paper, especially in regard to the careful choosing of the base field over which we work in our examples.


\renewcommand{\baselinestretch}{1.0}
\small\normalsize 

\bibliographystyle{amsalpha}

\end{document}